\documentclass{birkjour}
\usepackage{amsmath, amssymb, amsthm}

\newcommand
{\Dm}{\mathrm{Dom}}

\newcommand
{\spn}{\mathrm{span}}

\newcommand
{\tr}{\mathrm{tr}}

\newcommand
{\R}{{\mathbf R}}

\newcommand
{\ir}{\mathrm{i}}

\newcommand
{\Cc}{{\mathbf C}}

\newcommand
{\B}{\mathcal{H}}

\newcommand
{\Bc}{\mathcal{B}}

\newcommand
{\X}{\mathcal{T}}

\newcommand
{\V}{\mathcal{O}}

\newcommand
{\Rc}{\mathcal{R}}

\newcommand
{\p}{\mathrm{supp}}

\newcommand
{\Dis}{\mathrm{dist}}

\newcommand
{\ms}{\mathrm{mes}}

\newcommand
{\x}{\mathcal {M}}

\newcommand
{\Ec}{\mathcal{E}}

\newtheorem{theorem}{Theorem}[section]
\newtheorem{proposition}[theorem]{Proposition}
\newtheorem{lemma}[theorem]{Lemma}
\newtheorem{corollary}[theorem]{Corollary}
\theoremstyle{definition}
\newtheorem{definition}[theorem]{Definition}
\theoremstyle{remark}
\newtheorem{remark}[theorem]{Remark}
\numberwithin{equation}{section}

\renewcommand{\theequation}{\thesection.\arabic{equation}}

\begin{document}
	
\title[Holomorphy with respect to coupling constant] {Criterion of holomorphy with respect\\ to a coupling constant of
continuous
\\functions of a perturbed self-adjoint\\ operator}

\author[L. Zelenko]
{ Leonid Zelenko} 

\address{%
	Department of Mathematics \\
	University of Haifa  \\
	Haifa 31905  \\
	Israel}
\email{zelenko@math.haifa.ac.il}

\begin{abstract}
Sufficient and necessary conditions
on the spectral measure of a self-adjoint operator $A$, acting in a Hilbert space $\B$, are obtained,
under which for any continuous scalar function the operator function $\phi(A+\gamma B)$ is holomorphic with rspect to the coupling constant $\gamma$ in a neighborhood of $\gamma=0$, where $B$ is a self-adjoint operator. The sharpest results are obtained in the case where $B$ is a rank-one operator.
\end{abstract}

\subjclass{Primary 47A55, 47A56; \\Secondary 81Q15}

\keywords{Perturbation of a linear operator, function of a linear \\operator,
 holomorphic operator function}

\maketitle

\tableofcontents

\section{Introduction}\label{sec:introduction}
\setcounter{equation}{0}

In \cite{Ar-Z1} sufficient and necessary conditions
on the spectral measure of a self-adjoint operator $A$, acting in a Hilbert space $\B$, were obtained,
under which
any continuous function $\phi$ (without any additional smoothness
properties) has a directional operator - derivative
$$
\phi^{\prime}(A)(B):= \frac{\partial}{\partial \gamma}
\; \phi(A+\gamma B)_{|\gamma = 0}
$$
in the direction of a quite general bounded, self-adjoint operator $B$.
Here the operator function $\phi(A+\gamma B)$ is defined on the real
axis $\R$.
The sharpest results in \cite{Ar-Z1} were obtained in the case
where $B$ is a rank-one operator.
It turned out that in this case the sufficient condition for the existence
of the directional derivative for any continuous
function $\phi$ is the membership of the Borel transform
of the spectral measure of the operator A to the Hardy class $H_\infty$ (BTB property). Furthermore, in this case the directional derivative is expressed via the commutator of the multiplication operator $M_\phi$ on the function $\phi$ with the Hilbert transform $H$.

It is shown in the present paper by means of Friedrix method \cite{F} that the BTB property is sufficient and {\it necessary} for the existence of the holomorphic continuation
of the operator function $\phi(A+\gamma B)$ from an interval
$(-\delta,\delta)$ to the open disk
$D_\delta=\{\gamma\in\Cc\;\vert\;|\gamma|<\delta\}$ for any continuous
function $\phi$. Moreover, under this condition
there exists a family of unitary operators
$U_\gamma\;(\gamma\in(-\delta,\delta))$, which is a real-analytic operator
function on $(-\delta,\delta)$ and establishes the unitary equivalence
between $A+\gamma B$ and $A$ for each fixed $\gamma\in(-\delta,\delta)$.
 Let us notice that perturbations considered in this paper belong to the class of so called {\it gentle perturbations}, studied in \cite{F} and \cite{Rei}.

\section{Notations}\label{sec:Notations}

\vskip 0.3truecm


$\R$ and $\Cc\;$ are the fields of real and
complex numbers;

\noindent $\Re z$ and $\Im z$ are the real and the imaginary parts
of a number $z\in\Cc$;

\noindent$\V(x_0)$ is a neighborhood of a point $x_0$ belonging to
a topological space $\X$;

\noindent$\ms(A)$ is the Lebesgue measure of a measurable set
$A\subset \R$;

\noindent$\hat f$
is the Fourier transform of a function $f$ from $L_1(\R)$
or from $L_2(\R)$;

\noindent$\p(f)$ is the support of a function $f$;

\noindent If $\rho $ is a measure on $\R$, then $\p(\rho)$ denotes its support;

\noindent$M_\phi$ is the operator of multiplication by a function $\phi(t)$
on $\R$;

\noindent$\|\cdot\|_\Ec$ is the norm of  a Banach space $\Ec$;

\noindent$C[a,b]$ is the Banach algebra of continuous functions on
$[a,b]$ with the supremum norm;

\noindent$(\cdot,\cdot)_\B$ and $\|\cdot\|_\B$ are the inner product and
the norm in a Hilbert space $\B$.

\noindent If it is clear what
Hilbert space is meant, we shall simply write $(\cdot,\cdot)$ and
$\|\cdot\|$;

\noindent $\spn(\x)$ is the closure of the linear span of a subset $\x$ of
a Hilbert space $\B$;

\noindent For Banach spaces $E$ and $F$, $\Bc(E, F)$ is the Banach space of
all bounded linear operators from $E$ into $F$.

\noindent $\Bc(\B)=\Bc(\B, \B)$ is the Banach algebra of
bounded linear operators acting in a Hilbert space $\B$.

\vskip 0.3truecm

\noindent If $A$ is a linear operator
acting in a Hilbert space $\B$, then:
\vskip2mm

\noindent $\sigma(A)$ and $\sigma_e(A)$ are the spectrum and
essential spectrum
of $A$;

\noindent $\sigma_d(A)=\sigma(A)\setminus\sigma_e(A)$; if $A$ is self-adjoint,
$\sigma_d(A)$ consists of isolated eigenvalues of finite multiplicity;

\noindent$R_\lambda(A)\;(\lambda\notin\sigma(A))$ is the resolvent of $A$, i.e.,
$R_\lambda(A)=(A-\lambda I)^{-1}$;

\noindent$\tr(A)$ is the
trace of a linear operator $A$ belonging to the trace class.

\section{Properties of Holomorphic Unitary Equivalence,\\
Functional Holomorphy and Complex Local Isospectrality}
\label{sec:properties}
\setcounter{equation}{0}

In this section we consider self-adjoint bounded operators $A$ and $B$
acting in a Hilbert
space $\B$. Consider also a perturbed operator
$A_\gamma=A+\gamma B$, where $\gamma$ is a real coupling constant.
First we introduce some definitions. First of all, notice that the
holomorphy of an operator function
$F:\;U\rightarrow\Bc(\B)$ ($U\subseteq\Cc)$, $U$ is open) in the
uniform operator topology is equivalent to its holomorphy in
the weak operator topology (see \cite{D-Sc}). Hence we shall use
simply the term "holomorphic operator function".

\subsection{Definitions}

\begin{definition}\label{dfhlunt}
We say that the pair of operators $A,\;B$ has the property of
Holomorphic Unitary Equivalence (briefly-HUE property), if there
exists $\delta>0$ and a family of unitary operators
$\{U_\gamma|\}_{\gamma\in(-\delta,\delta)}$ such that
\begin{equation}\label{uneq}
\forall\;\gamma\in(-\delta,\delta):\;A_\gamma=(U_\gamma)^{-1}A
U_\gamma
\end{equation}
and the operator function $U_\gamma$
admits a holomorphic continuation from $(-\delta,\delta)$ to the
open disk $D_\delta=\{\gamma\in\Cc\;\vert\;|\gamma|<\delta\}$.
\end{definition}

\begin{definition}\label{dfFH}
We say that the pair of operators $A,\;B$ has the property of
Functional Holomorphy
(briefly- FH  property),
if
there exists $\delta>0$ such that for any continuous function
$\phi:\;\R\rightarrow\Cc$ the operator function
$\phi(A_\gamma)$
admits a holomorphic continuation $\Phi (\gamma)$ from
$(-\delta,\delta)$ to the open disk
$D_\delta=\{\gamma\in\Cc\;\vert\;|\gamma|<\delta\}$.
\end{definition}

\begin{definition}\label{dfLSS}
We say that the pair of operators $A,\;B$ has the property of
Complex Local Isospectrality (briefly- CLI property), if there
exists $\delta>0$ such that
\begin{equation}\label{spceq}
\forall\;\gamma\in D_\delta:\quad \sigma(A_\gamma)=\sigma(A).
\end{equation}
\end{definition}

\subsection{Relation between the above properties}

\begin{proposition}\label{pHUEimpl}
The HUE property of the pair $A,\;B$ implies its FH and CLI
properties.
\end{proposition}
\begin{proof}
Let $V_\gamma$ be the holomorphic continuation of the operator
function $U_\gamma$ from $(-\delta, \delta)$ to the open disk
$D_\delta$. Then there exists $\sigma\in(0,\delta]$ such that for
each $\gamma\in D_\sigma$ operator $V_\gamma$ is continuously
invertible. Then, by the principle of holomorphic continuation,
(\ref{uneq}) implies that $A_\gamma=(V_\gamma)^{-1}A V_\gamma$ for
any $\gamma\in D_\sigma$, that is the operator $A_\gamma$ is
similar to $A$ for these values of $\gamma$. This fact implies the
CLI property of the pair $A,\;B$. The FH property follows from the
evident equality $\phi(A_\gamma)=(V_\gamma)^{-1}\phi(A) V_\gamma$
which is valid for any continuous function
$\phi:\;\R\rightarrow\Cc$. The proposition is proven.
\end{proof}

Before formulating the next theorem, we need the following
lemma, which is a slight generalization of Proposition 3.5 from
\cite{Ar-Z1}.
\begin{lemma}\label{linc}
Let A and B be linear operators acting in a Hilbert space $\B$,
$A$ is self-adjoint and B is compact. Assume that for some $\delta>0$
\begin{equation}\label{spcinc}
\forall\;\gamma\in D_\delta:\quad \sigma(A_\gamma)\subseteq\sigma(A).
\end{equation}
Then the property (\ref{spceq}) is valid.
\end{lemma}
\begin{proof}
Since $B$ is compact, then
by H. Weyl Theorem (\cite{Kat}, Chapter IV, \S 5, $n^o$ 6,
Theorem 5.35),
\begin{equation}
\label{sgme}
\forall\;\gamma\in\Cc:\quad
\sigma_e(A_\gamma)=\sigma_e(A).
\end{equation}
Then, in view of (\ref{spcinc}), in order to prove (\ref{spceq}),
it is enough to show that
\begin{equation}
\label{dlta}
\sigma_d(A)\subseteq \sigma_d(A_\gamma)\quad
\forall \gamma\in D_\delta.
\end{equation}
Take $\lambda_0\in \sigma_d(A)$. Then, in view of
(\ref{spcinc}), there exist$(-\delta,\delta)$s a neighborhood
$D_\epsilon(\lambda_0)=\{\lambda\in\Cc\;\vert\;|
\lambda-\lambda_0|<\epsilon\}$ of the point $\lambda_0$,
such that
$(D_\epsilon(\lambda_0)\setminus\{\lambda_0\})
\cap\sigma(A_\gamma)=\emptyset$ for any
$\gamma\in D_\delta$
Since the function
\begin{equation*}
T(\gamma)=-\frac{1}{2\pi \mathrm{i}}
\tr\left(\oint_{|\lambda-\lambda_0|=\frac{\epsilon}{2}}R_\lambda(A_\gamma)\,
\mathrm{d}\lambda\right)
\end{equation*}
is continuous and takes only non-negative integer values and $T(0)>0$,
then $T(\gamma)>0$ for any
$\gamma\in (D_\delta$.
This means that for these values of
$\gamma$ the point $\lambda_0$ belongs to $\sigma_d(A_\gamma)$. So,
(\ref{dlta}) is valid, hence (\ref{spceq})
is valid too.
\end{proof}

We now turn to the main result of this section.

\begin{theorem}\label{tFHLSC}
Let $A,\;B$ be a pair of bounded self-adjoint operators acting in
a Hilbert space $\B$. If this pair has the FH property and $B$ is
compact, then it has the CLI property.
\end{theorem}
\begin{proof}
Let $M>0$ be such a number that $\sigma(A)\subset(-M,M)$. Then for a small
enough $\gamma$ $\sigma(A_\gamma)\subset(-M,M)$. In view of the FH
property, for any function $\phi\in C[-M,M]$ the
operator function $\Phi(\gamma)$, which is the holomorphic
continuation of $\phi(A_\gamma)$ from $(-\delta, \delta)$ to the disk
$D_\delta$, is differentiable at the point $\gamma=0$ with respect to the uniform
operator topology, that is in this topology there exists the limit
$$
\lim_{\gamma\rightarrow}\frac{\Phi(\gamma)-\phi(A)}
{\gamma}
$$
for any $\phi\in C[-M,M]$. By Banach-Steinhaus Theorem this fact implies
that
\begin{equation}\label{lmspSg}
\limsup_{\gamma\rightarrow 0}\left\|
\Sigma(\cdot,\gamma)\right\|_{\Bc\left(C[-M,M],\Bc(\B)\right)}<\infty,
\end{equation}
where
\begin{equation}\label{dfSg}
\Sigma(\phi,\gamma)=\frac{\Phi(\gamma)-\phi(A)}
{\gamma}
\end{equation}


By Lemma \ref{linc}, in  order to prove CLI property, it is enough
to establish the property:
\begin{equation}\label{inclspc}
\exists\;\sigma>0\;\;
 \forall\;\gamma\in D_\sigma:\quad
 \sigma(A_\gamma)\subseteq\sigma(A).
\end{equation}
Assume, on the contrary, that this property does not hold. This means that
 there exists a sequence $\{\gamma_n\}_{n=1}^\infty$ such that
$\lim_{n\rightarrow\infty}\gamma_n=0$, $\gamma_n\neq 0$
and
\begin{equation}\label{dfSn}
S_n=\sigma(A_{\gamma_n})\setminus\sigma(A)\neq\emptyset
\end{equation}
for all natural $n$. Since $B$ is compact, then by H. Weyl Theorem,
(\ref{sgme}) is valid. Hence each set
$S_n$ is contained in $\sigma_d(A_{\gamma_n})$, that is it consists of
isolated eigenvalues of the operator $A_{\gamma_n}$
of finite algebraic multiplicity.
If we shall prove that
\begin{equation}\label{lmspSginf}
\limsup_{n\rightarrow\infty}\left\|
\Sigma(\cdot,\gamma_n)\right\|_{\Bc\left(C[-M,M],\Bc(\B)\right)}=\infty,
\end{equation}
then we shall obtain a contradiction with (\ref{lmspSg}) and the theorem
will be proved. Denote by $\Ec$ the set of all entire functions
$\phi:\;\Cc\rightarrow\Cc$ and for any $\phi\in\Ec$ denote
\begin{equation}\label{dfphir}
\phi_r=\phi\vert_{[-M,M]}.
\end{equation}
Then, in order to prove (\ref{lmspSginf}),
it is enough to show that
\begin{equation}\label{lmspSgent}
\limsup_{n\rightarrow\infty}\left(\sup\left\{\left\|
\Sigma(\phi_r,\gamma_n)\right\|_{\Bc(\B)}\;:\;\phi\in\Ec,\;
\|\phi_r\|_{C[-M,M]}\le 1
\right\}\right)=\infty.
\end{equation}
As is easy to see, the family of spectra
$\{\sigma(A_\gamma)\}_{\gamma\in D_\delta}$ is uniformly bounded in the
complex plane. Then there exists a closed contour $\Gamma$ lying in the
resolvent set $\Rc(A_\gamma)$ of each operator
$A_\gamma\;(\gamma\in D_\delta)$ and surrounding all the spectra
$\sigma(A_\gamma)\;(\gamma\in D_\delta)$. As is known, for any
$\phi\in\Ec$
$$
\phi(A_\gamma)=-\frac{1}{2\pi \mathrm{i}}\oint_\Gamma
\phi(\lambda)R_\lambda(A_\gamma)\,d\lambda.
$$
and for any fixed
$\lambda\in\Rc(A_\gamma)$ the resolvent $R_\lambda(A_\gamma)$
is a holomorphic
operator function of $\gamma$ in a neighborhood of $\gamma=0$.
These circumstances imply
that $\phi(A_\gamma)$ is a holomorphic operator function of
$\gamma$ in a disk $D_\sigma\;(\sigma\in(0,\delta))$.
Thus, by
the principle of analytic continuation, for any $\phi\in\Ec$ the
function $\Phi(\gamma)$, corresponding to $\phi_r$, coincide
with $\phi(A_\gamma)$ in $D_\sigma$. Using the assumption (\ref{dfSn}),
we can take $\lambda_n\in S_n$ for each natural $n$. We can select from
this sequence an infinite subsequence (we denote it in the same manner:
$\lambda_n$) having one of the following properties:

(a) or $\lambda_n\notin\R$ for any natural $n$,

(b) or $\lambda_n\in\R$
for any natural $n$.


Let us construct a sequence of entire functions $\phi_n(\lambda)$
in the following manner. In case (a) we put:
\begin{eqnarray}\label{dfphin1}
\phi_n(\lambda):=\left\{\begin{array}{ll}
e^{-\mathrm{i}\tau_n(\lambda-\lambda_n)},&\;\mathrm{if}\;\;
\Im(\lambda_n)>0,\\
e^{\mathrm{i}\tau_n(\lambda-\lambda_n)},&\;\mathrm{if}\;\;
\Im(\lambda_n)<0.
\end{array}\right.
\end{eqnarray}
In case (b) we put:
\begin{equation}\label{dfphin3}
\phi_n(\lambda):=e^{-\tau_n(\lambda-\lambda_n)^2}.
\end{equation}
In both the cases the sequence $\tau_n>0$ will be specified
in the sequel. For each natural $n$ let as take an eigenvector
$e_n$ of the operator $A_{\gamma_n}$, corresponding to its eigenvalue
$\lambda_n$, such that $\|e_n\|=1$.
As is known, $\phi_n(A_{\gamma_n})e_n=\phi(\lambda_n)e_n$. Taking
into account that, in view of definitions (\ref{dfphin1}),
(\ref{dfphin3}), $\phi_n(\lambda_n)=1$, we obtain in both the cases (a)
and (b):
$$
\Sigma(\phi_{n,r},\gamma_n)e_n=\frac{\phi_n(A_{\gamma_n})e_n-\phi_n(A)e_n}
{\gamma_n}=\frac{1}{\gamma_n}e_n-\frac{1}{\gamma_n}\phi_n(A)e_n.
$$
Here $\phi_{n,r}$ is defined by (\ref{dfphir}) with $\phi=\phi_n$.
Then
\begin{equation}\label{estblw}
\Vert\Sigma(\phi_{n,r},\gamma_n)\Vert_{\Bc(\B)}\ge\frac{1}{|\gamma_n|}-
\frac{1}{|\gamma_n|}\Vert\phi_n(A)\Vert_{\Bc(\B)}.
\end{equation}
On the other hand, since the operator $A$ is self-adjoint, we get:
\begin{equation}\label{nrphin}
\Vert\phi_n(A)\Vert_{\Bc(\B)}\le\max_{\lambda\in\sigma(A)}
|\phi_n(\lambda)|.
\end{equation}
In the case (a) we obtain from (\ref{dfphin1}):
\begin{equation}\label{estnrph1}
\max_{\lambda\in\sigma(A)}|\phi_n(\lambda)|\le e^{-\tau_n|\Im(\lambda_n)|},
\end{equation}
and in the case (b) we obtain from (\ref{dfphin3}):
\begin{equation}\label{estnrph2}
\max_{\lambda\in\sigma(A)}|\phi_n(\lambda)|\le e^{-\tau_nd_n^2},
\end{equation}
where $d_n=\Dis(\lambda_n,\sigma(A))$. Let us choose $\tau_n$ in the
following manner: in case (a) we put $\tau_n=|\Im(\lambda_n)|^{-1}$
and in case (b) we put $\tau_n=d_n^{-2}$. Then in both cases (a)
and (b) we obtain from (\ref{estblw})-(\ref{estnrph2}):
\begin{equation}\label{finest}
\Vert\Sigma(\phi_{n,r},\gamma_n)\Vert_{\Bc(\B)}\ge\frac{1-e^{-1}}
{|\gamma_n|}
\end{equation}
Observe that, in view of definitions (\ref{dfphin1}), (\ref{dfphin3}),
$\Vert\phi_{n,r}\Vert_{C[-M,M]}\le 1$. Then, taking into accont that
$\lim_{n\rightarrow\infty}\gamma_n=0$, we obtain from (\ref{finest}) the desired
limiting relation (\ref{lmspSgent}). The theorem is proven.
\end{proof}

\section{The case of a rank-one perturbation}
\label{sec:rank one}
\setcounter{equation}{0}

In the space $\B=L_2(\R,\rho)$ with a compactly supported non-negative
Borel measure $\rho$ consider the multiplication operator
$$
(Af)(\mu)=\mu f(\mu)\;\;(f\in\B)
$$
and its perturbation $A_\gamma=A+\gamma B$, where $B$ is a rank-one
operator: $B=(\cdot,g)g$ and $g(\mu)=1$ almost everywhere with respect  to
the measure $\rho$ (hence $g\in\B$). We call $\rho$ the
{\it spectral measure} of $A$. Let us take $M>0$ such that
$\sigma(A)=\p(\rho)\subset(-M,M)$.
In this section we shall obtain
a necessary and sufficient condition for the spectral measure $\rho$
ensuring the UHE property of the pair $A,\;B$.
\vskip 0.3truecm

\subsection{Friedrichs method}

We shall use
the method of 0. K. Friedrichs (\cite{F}, Chapt. II, Sect. 6).
Let us describe it briefly.
One searches for a pair of operators $U^+$ and $U^-$ satisfying
the conditions
\begin{equation}\label{eqpl}
AU^+=U^+A_\gamma
\end{equation}
and
\begin{equation}\label{eqmn}
A_\gamma U^-=U^-A
\end{equation}
and shows that they can be chosen such that
$U^+U^-=U^-U^+=I$ and $U^-=(U^+)^\star$. Let us write the equations
(\ref{eqpl}) and (\ref{eqmn}) in the form:
\begin{equation}\label{eqpl1}
[A,U^+]=AU^+-U^+A=\gamma U^+B,
\end{equation}
\begin{equation}\label{eqmn1}
[A,U^-]=AU^--U^-A=-\gamma BU^-.
\end{equation}
In order to solve these equations, we need be able to solve a simpler
equation
\begin{equation}\label{commeq}
[A,Z]=AZ-ZA=R\;\;(R\in\Bc(\B))
\end{equation}
in the class $\Bc(\B)$. It is clear that if there exists a solution
$Z_0$ of this equation, it is not unique, because $Z=Z_0+C$, where
$C$ commutes with $A$, is also solution of this equation. We shall
show that equation (\ref{commeq}) has a solution, if $R$ belongs to some
class of operators $\Rc$. Denote by $\Gamma$ a transformer which
associates a solution $Z$ of equation (\ref{commeq}) to each $R\in\Rc$:
$Z=\Gamma R$. In the sequel this transformer $\Gamma$ will be chosen in a
suitable manner. Following the method of Friedrichs, one searches for
solutions of equations (\ref{eqpl1}), (\ref{eqmn1}) in the form:
\begin{equation}\label{eqUplmn}
U^+=I+\Gamma R^+,\quad U^-=I-\Gamma R^-,
\end{equation}
where $R^+$ and $R^-$ satisfy the equations
\begin{equation}\label{eqRpl}
R^+=\gamma(I+\Gamma R^+)B,
\end{equation}
\begin{equation}\label{eqRmn}
R^-=\gamma B(I-\Gamma R^-).
\end{equation}
We see easily that if $R^+$ and $R^-$ are solutions of these equations, then
$U^+$ and $U^-$, defined by (\ref{eqUplmn}), indeed satisfy
equations (\ref{eqpl1}), (\ref{eqmn1}).

\subsection{Transformer $\Gamma$}

In order to define the
transformer $\Gamma$, we need to solve the equation (\ref{commeq}).
How to do this ?  In the book \cite{D-K} the operator equations of the
form
\begin{equation}\label{eqcomAB}
AZ-BZ=R\;\;(A,B\in\Bc(\B))
\end{equation}
(and even of more general form) are considered. As is shown there, if
\begin{equation}\label{spcond}
\sigma(A)\cap\sigma(B)=\emptyset,
\end{equation}
then the equation (\ref{eqcomAB}) has a unique solution in $\Bc(\B)$
and it has the form:
$$
Z=-\frac{1}{4\pi^2}\oint_{C_A}\oint_{C_B}\frac{R_\lambda(A) R R_\mu(B)}
{\lambda-\mu}\,d\lambda\,d\mu,
$$
where the contours $C_A$ and $C_B$ lie in $\Rc(A)\cap\Rc(B)$, $C_A$
surrounds $\sigma(A)$, but does not surround $\sigma(B)$ and $C_B$
surrounds $\sigma(B)$, but does not surround $\sigma(A)$. We see that
the condition (\ref{spcond}) is not satisfied for our equation
(\ref{commeq}). Hence we consider first the following ``regularized''
equation
\begin{equation}\label{regeq}
(A+\mathrm{i}\epsilon I)Z-ZA=R,
\end{equation}
where $\epsilon>0$. Since $A$ is self-adjoint,
$\sigma(A+\mathrm{i}\epsilon I)\cap\sigma(A)=\emptyset$, hence the equation
(\ref{regeq}) has in $\Bc(\B)$ a unique solution
\begin{equation}\label{dfGmep}
Z_\epsilon=\Gamma_\epsilon R=
-\frac{1}{2\pi\mathrm{i}}\oint_{C_\epsilon^1}\left(
-\frac{1}{2\pi\mathrm{i}}\oint_{C_\epsilon^2}
\frac{R_\lambda(A+\mathrm{i}\epsilon I)}
{\lambda-\mu}\,d\lambda\right)R R_\mu(A)\,d\mu,
\end{equation}
where the contours $C_\epsilon^1$ and $C_\epsilon^2$
lie in $\Rc(A)\cap\Rc(A+\mathrm{i}\epsilon I)$, $C_\epsilon^2$
surrounds $\sigma(A+\mathrm{i}\epsilon I)$, but does not surround $\sigma(A)$ and
$C_\epsilon^1$
surrounds $\sigma(A)$, but does not surround
$\sigma(A+\mathrm{i}\epsilon I)$. Hence the expression inside the brackets
in (\ref{dfGmep}) yields $f_\mu(A+\mathrm{i}\epsilon I)$, where
$f_\mu(\lambda)=(\lambda-\mu)^{-1}$. Thus, we have:
\begin{equation}\label{dfGmep1}
Z_\epsilon=\Gamma_\epsilon R=
-\frac{1}{2\pi\mathrm{i}}\oint_{C_\epsilon^1}R_{\mu-\mathrm{i}\epsilon}(A)
R R_\mu(A)\,d\mu.
\end{equation}
If for some $R\in\Bc(\B)$ there exists the limit
\begin{equation}\label{lim}
Z=\lim_{\epsilon\downarrow 0}Z_\epsilon
\end{equation}
in the strong operator topology,
then tending $\epsilon\downarrow 0$ in (\ref{regeq}) with $Z=Z_\epsilon$,
we obtain that
$Z$ is a solution of the equation (\ref{commeq}) belonging to $\Bc(\B)$.
As above, we denote it by $\Gamma R$. Denote by $\Dm(\Gamma)$ the domain
of definition of the transformer $\Gamma$, i. e. this is the set
of all $R\in\Bc(\B)$ such that for the solution $Z_\epsilon$ of the
equation (\ref{regeq}) there exists the limit (\ref{lim}) in the strong
operator topology. It is clear that $\Gamma$ is a linear transformer.

\subsection{Transformer $\Gamma$ and Riesz projection}

 Let us return to the multiplication operator $A$ and its rank-one
perturbation described in the beginning of this section.
In the sequel we shall describe a class of rank-one operators acting
in the space $\B=L_2(\R,\rho)$ and belonging to $\Dm(\Gamma)$.
To this end we need the following

\begin{lemma}\label{lGmep}
Assume that the spectral measure $\rho$ of the operator $A$ is absolutely
continuous with the density $\tilde\rho$. Let $R$ be a rank-one
operator acting in $\B$, that is
$R=(\cdot,r_1)r_2\;(r_1,r_2\in\B)$, and let $\Gamma_\epsilon\;(\epsilon>0)$
be the transformer in $\Bc(\B)$ defined by (\ref{dfGmep1}).
Then $\Gamma_\epsilon R$ is an integral operator with the kernel
$$
\frac{r_2(x)\overline{r_1(t)}\tilde\rho(t)}{x+\mathrm{i}\epsilon-t}.
$$
\end{lemma}
\begin{proof}
Recall that $\sigma(A)=\p(\rho)\subset(-M,M)$. Let us take in
(\ref{dfGmep1})
$$
C_\epsilon^1=\left\{\lambda\in\Cc\;\vert\;
\Dis(\lambda,(-M,M))=\frac{\epsilon}{4}\right\}.
$$
We have for $f\in\B$:
\begin{eqnarray}
&&\Gamma_\epsilon R f(x)=-\frac{1}{2\pi\mathrm{i}}\oint_{C_\epsilon^1}
\,d\mu\int_{-\infty}^\infty\frac{r_2(x)f(t)
\overline{r_1(t)}\tilde\rho(t)\,dt}
{(x+\mathrm{i}\epsilon-\mu)(t-\mu)}=
\nonumber\\
&&
\int_{-M}^M f(t)
\overline{r_1(t)}\tilde\rho(t)\,dt\left(\frac{1}{2\pi\mathrm{i}}
\oint_{C_\epsilon^1}
\frac{d\mu}{(x+\mathrm{i}\epsilon-\mu)(\mu-t)}\right)r_2(x)
\nonumber
\end{eqnarray}
Using Cauchy formula inside of the brackets, we get:
\begin{equation}\label{frmGmep}
\Gamma_\epsilon R f(x)=\int_{-\infty}^\infty\frac{r_2(x)
\overline{r_1(t)}\tilde\rho(t)f(t)\,dt}
{x+\mathrm{i}\epsilon-t}.
\end{equation}
This proves the lemma.
\end{proof}
We can write the formula (\ref{frmGmep}) with the help of the
following bounded operator acting in the space $L_2(\R)$:
\begin{equation}\label{pepl}
(P_{+,\epsilon}h)(u):=\frac{1}{2\pi \mathrm{i}}\int_{-\infty}^\infty
\frac{h(s)\,\mathrm{d}s}{s-u-\mathrm{i}\epsilon}
\quad(\epsilon>0,\; h\in L_2(\R)).
\end{equation}
In order to do this, we shall impose the following conditions on
the operator $A$ and on a rank-one operator $R$:
\vskip 0.2truecm

(A) The spectral measure $\rho$ of the operator $A$ is absolutely
continuous and its density $\tilde\rho$ belongs to the class
$L_\infty(\R)$;
\vskip 0.2truecm

(B) $R=(\cdot,r_1)r_2$, where $r_k\in\L_\infty(\R)\;(k=1,2)$.
\vskip 0.3truecm

We have the following consequence of Lemma \ref{lGmep}:
\begin{corollary}
If  the conditions (A) and (B) are satisfied, then
the rank-one operator $R=(\cdot,r_1)r_2$ is bounded in $\B$ and
the representation is valid:
\begin{equation}\label{rprGmep}
\Gamma_\epsilon R=-2\pi\mathrm{i}JM_{r_2}P_{+,\epsilon}
M_{\bar r_1\tilde\rho},
\end{equation}
where $J$ is the embedding operator from $L_2(\R)$ into $\B$.
\end{corollary}
\begin{proof}
Observe that the assumptions of the corollary and the fact that the
measure $\rho$ has a compact support imply that $r_k\in\B$, hence
$R$ is a bounded operator. Observe that since
$\tilde\rho\in L_\infty(\R)$, the embedding
$J:\;L_2(\R)\rightarrow\B$ holds and it is continuous. Furthermore,
in view of the above assumptions, the multiplication operators
$M_{\bar r_1\tilde\rho}$ and $M_{r_2}$ act continuously from
$\B$ into $L_2(\R)$ and from $L_2(\R)$ into itself, respectively.
These circumstances and formula (\ref{frmGmep}) imply the desires
representation (\ref{rprGmep}).
\end{proof}
\begin{proposition}\label{prGmR}
If the conditions (A) and (B) are satisfied, then
the rank-one operator $R=(\cdot,r_1)r_2$ belongs to $\Dm(\Gamma)$
and the representation is valid:
\begin{equation}\label{rprGmR}
\Gamma R=-2\pi\mathrm{i}JM_{r_2}P_+
M_{\bar r_1\tilde\rho},
\end{equation}
where $P_+$ is Riesz projection in $L_2(\R)$ on the Hardy space
\begin{equation*}
\B_+=\{f\in L_2(\R)\;|\;\hat f(\omega)=0\;\;{\rm a.e.}\;\;{\rm on}
\;\;(-\infty,0)\}.
\end{equation*}
\end{proposition}
\begin{proof}
By Proposition 2,1 of \cite{Ar-Z1}, the family of operators
$P_{+,\epsilon}$, defined by (\ref{pepl}), has the property:
$\lim_{\epsilon\downarrow 0}P_{+,\epsilon}=P_+$ in the strong operator
topology. Hence, in view of (\ref{rprGmep}),
$$
\lim_{\epsilon\downarrow 0}\Gamma_\epsilon R=-2\pi\mathrm{i}JM_{r_2}P_+
M_{\bar r_1\tilde\rho}
$$
in the strong operator topology. This means that $R\in\Dm(\Gamma)$ and
the desired equality (\ref{rprGmR}) is valid. The proposition is proven.
\end{proof}

\subsection{BTB property and solution of basic equations}

${\bf 3.4^o}$ We now turn to solution of the equations (\ref{eqRpl}) and
(\ref{eqRmn}). In the sequel we need the following notion.

\begin{definition}\label{dfBBT}
We say that the spectral measure $\rho$ of the operator $A$
has the property of Borel Transform Boundedness (briefly - BTB property),
if its Borel transform
\begin{equation}\label{dfBortr}
\mathrm{B}\rho(\lambda):=\int_{\R} \frac{\rho(\mathrm{d}t)}{t-\lambda}
\end{equation}
belongs to the Hardy class $H_\infty^+$ (that is, it is bounded in
the upper (hence in the lower) half-plane of the complex plane
$\Cc$).
\end{definition}

\begin{remark}\label{rBrtr}
BTB property implies that the measure $\rho$
is absolutely continuous and its density $\tilde\rho$ belongs to the
class $L_\infty(\R)$. This fact follows from Stieltyes inversion formula
$$
\rho(b)-\rho(a)=-\frac{1}{\pi}\lim_{\tau\downarrow 0}\int_a^b
\Im(\mathrm{B}\rho(t+\mathrm{i}\tau))dt,
$$
in which $\rho(\mu)$ is the non-decreasing
function  defining the measure $\rho$ and $a,\;b$ are points of
continuity of this function.
\end{remark}

As in the beginning of this section, along with the multiplication
operator $A$ acting in $\B=L_2(\R,\rho)$ we consider the rank-one
perturbing operator $B=(\cdot,g)g$, where $g(t)=1$ almost everywhere
with respect  to the measure $\rho$.
The following statement is valid.
\begin{proposition}\label{solRplmn}
Assume that the spectral measure $\rho$ of the operator $A$ has the
BTB property. Then for a small enough $\gamma\in\Cc$ the equations
(\ref{eqRpl}) and (\ref{eqRmn}) have unique solutions $R^+$ and $R^-$
in the class of operators satisfying condition (B).
Furthermore, they have the form:
\begin{equation}\label{solRpl}
R^+=\gamma(\cdot,g)M_{\psi_\gamma} g
\end{equation}
and
\begin{equation}\label{solRmn}
R^-=\gamma(\cdot,M_{\psi_\gamma} g)g,
\end{equation}
where
\begin{equation}\label{dfps}
\psi_\gamma(x)=(1+2\pi\mathrm{i}\gamma P_+\tilde\rho(x))^{-1}.
\end{equation}
\end{proposition}
\begin{proof}
Let $R^+$ and $R^-$ be solutions of equations (\ref{eqRpl}) and
(\ref{eqRmn}) satisfying condition (B). Then, taking into
account that $B=(\cdot,g)g$ and using Proposition \ref{prGmR}, we have:
\begin{equation}\label{eqRpl1}
R^+=\gamma(\cdot,g)(I+\Gamma R^+)g,
\end{equation}
\begin{equation}\label{eqRmn1}
R^-=\gamma(\cdot,(I-\Gamma R^-)^\star g)g,
\end{equation}
that is $R^+$ and $R^-$ have the form:
\begin{equation}\label{ans1}
R^+=(\cdot,g)r^+,
\end{equation}
\begin{equation}\label{ans2}
R^-=(\cdot,r^-)g,
\end{equation}
where $r^+,r^-\in\B$. let us find $r^+$ and $r^-$.
By Proposition \ref{prGmR},
\begin{equation}\label{GmRpl}
\Gamma R^+=-2\pi\mathrm{i}JM_{r^+}P_+
M_{\bar g\tilde\rho},
\end{equation}
and
\begin{equation}\label{GmRmn}
\Gamma R^-=-2\pi\mathrm{i}JM_gP_+
M_{\bar r^-\tilde\rho}.
\end{equation}
Hence, taking into account that $J^\star=M_{\tilde\rho}$, we have:
$$
(\Gamma R^-)^\star=2\pi\mathrm{i}JM_{r^-}P_+
M_{\bar g\tilde\rho}.
$$
Then, after substituting (\ref{ans1}), (\ref{ans2}), (\ref{GmRpl}) and
(\ref{GmRmn}) into (\ref{eqRpl1}) and (\ref{eqRmn1}), we get:
$$
(\cdot,g)r^+=\gamma(\cdot,g)\left(g-2\pi\mathrm{i}JM_{r^+}
P_+(\vert g\vert^2\tilde\rho)\right)
$$
and
$$
(\cdot,r^-)g=\gamma(\cdot,g-2\pi\mathrm{i}JM_{r^-}
P_+(\vert g\vert^2\tilde\rho))g,
$$
that is
\begin{equation}\label{eqrpl}
r^+(x)\left(1+2\pi\mathrm{i}\gamma P_+(\vert g\vert^2\tilde\rho)(x)\right)=
\gamma g(x)
\end{equation}
and
\begin{equation}\label{eqrmn}
r^-(x)\left(1+2\pi\mathrm{i}\gamma P_+(\vert g\vert^2\tilde\rho)(x)\right)=
\gamma g(x).
\end{equation}
Observe that $\vert g\vert^2\tilde\rho=\tilde\rho$, because $g(x)=1$ almost
everywhere with respect  to the measure $\rho$. Consider the operator
$P_{+,\epsilon}\;(\epsilon>0)$ defined by (\ref{pepl}). In view of
BTB property of $\rho$, the family of functions
$\{P_{+,\epsilon}\tilde\rho(x)\}_{\epsilon>0}$ is uniformly bounded on $\R$.
On the other hand, by the property of Hardy class $H_2$,
\begin{equation}\label{limPplep}
\lim_{\epsilon\downarrow 0}P_{+,\epsilon}\tilde\rho(x)=
P_+\tilde\rho(x)
\end{equation}
for almost all $x\in\R$. These circumstances mean that
$P_+\tilde\rho\in L_\infty(\R)$. Hence for a small enough $\gamma\in\Cc$ the
function $\psi_\gamma$, defined by (\ref{dfps}), belongs to the class
$L_\infty(\R)$. Then we have from (\ref{eqrpl}) and (\ref{eqrmn}) that
\begin{equation}\label{rpleqrmn}
r^+=r^-=M_{\psi_\gamma}g.
\end{equation}
Hence, in view of (\ref{ans1}) and
(\ref{ans2}), the operators $R^+$ and $R^-$ have the form
(\ref{solRpl}) and (\ref{solRmn}). So, we have proved the uniqueness
of solution of equations (\ref{eqRpl}) and (\ref{eqRmn}) in the class of
operators satisfying condition (B) (for a small enough $\gamma$).
Carrying out the above arguments in the inverse direction, we can show
that for a small enough $\gamma$ the operators $R^+$ and $R^-$,
expressed by (\ref{solRpl}) and (\ref{solRmn}), satisfy condition (B) and
equations (\ref{eqRpl}) and (\ref{eqRmn}), respectively. The proposition
is proven.
\end{proof}
\vskip 0.3truecm

 Let us return to formulas (\ref{eqUplmn}) for the operators
$U_\gamma^+$ and $U_\gamma^-$. In this section we shall show that
$U_\gamma^+U_\gamma^-=U_\gamma^-U_\gamma^+=I$ for a small enough $\gamma$.
Substituting $R^+$ and $R^-$ from (\ref{solRpl}) and (\ref{solRmn})
into (\ref{eqUplmn}) and taking into account (\ref{GmRpl}), (\ref{GmRmn})
and (\ref{rpleqrmn}), we obtain the explicit formulas for
$U_\gamma^+$ and $U_\gamma^-$:
\begin{equation}\label{frmUpl}
U_\gamma^+=I-2\pi\mathrm{i}\gamma JM_{\psi_\gamma\cdot g}P_+
M_{\bar g\cdot\tilde\rho}
\end{equation}
\begin{equation}\label{frmUmn}
U_\gamma^-=I+2\pi\mathrm{i}\gamma JM_gP_+
M_{\bar\psi_\gamma\cdot\bar g\cdot\tilde\rho},
\end{equation}
where $\psi_\gamma$ is defined by (\ref{dfps}).
In the sequel we need the following

\begin{lemma}\label{prpGmep}
The following identity is valid for any $R_1,R_2\in\Bc(\B)$
and $\epsilon>0$ (see \cite{F}, Chapter II, Sect. 6):
\begin{equation}\label{mainprp}
\Gamma_\epsilon R_1\cdot\Gamma_\epsilon R_2=
\Gamma_{2\epsilon}(\Gamma_\epsilon R_1\cdot R_2+
R_1\cdot\Gamma_\epsilon R_2).
\end{equation}
\end{lemma}
\begin{proof}
Denote
$$
[A,Z]_\epsilon= (A+\ir\epsilon I)Z-AZ.
$$
Then, by the definition of $\Gamma_\epsilon$, the equality
$[A,Z]_\epsilon=R$ is equivalent to the equality $Z=\Gamma_\epsilon R$.
Let us take $R_1,R_2\in\Bc(\b)$ and denote
$Z_k=\Gamma_\epsilon R_k\;(k=1,2)$. We have:
\begin{eqnarray}
&&\hskip-5mm[A,Z_1Z_2]_{2\epsilon}=(A+2\ir\epsilon I)Z_1Z_2-Z_1Z_2A=
(A+\ir\epsilon I)Z_1Z_2 +\ir\epsilon Z_1Z_2-
\nonumber\\
&&\hskip-5mm Z_1AZ_2+Z_1AZ_2-Z_1Z_2A=((A+\ir\epsilon I)Z_1-Z_1A)Z_2+
Z_1((A+\ir\epsilon I)Z_2-
\nonumber\\
&&\hskip-5mm Z_2A)=[A,Z_1]_\epsilon Z_2-Z_1[A,Z_2]_\epsilon.
\nonumber
\end{eqnarray}
So, we get:
$$
[A,Z_1Z_2]_{2\epsilon}=[A,Z_1]_\epsilon Z_2-Z_1[A,Z_2]_\epsilon.
$$
Applying $\Gamma_\epsilon$ to both sides of the latter equality, we
obtain the desired identity (\ref{mainprp}). The lemma is proven.
\end{proof}

On the base of the previous lemma we obtain the following

\begin{lemma}\label{prpGm}
If all the conditions of Proposition \ref{solRplmn} are satisfied,
then the operators $R^+$ and $R^-$, defined by (\ref{solRpl}) and
(\ref{solRmn}), have the property:
\begin{equation}\label{mainprp1}
\Gamma R^+\cdot\Gamma R^-=
\Gamma(\Gamma R^+\cdot R^-+
R^-\cdot\Gamma R^+).
\end{equation}
\end{lemma}
\begin{proof}
By Lemma \ref{prpGmep}, we have:
$$
\Gamma_\epsilon R^+\cdot\Gamma_\epsilon R^-=
\Gamma_{2\epsilon}(\Gamma_\epsilon R^+\cdot R^-+
R^-\cdot\Gamma_\epsilon R^+).
$$
On the other hand, since the operators $R^+$ and $R^-$ satisfy
the condition
(B), then by Proposition \ref{prGmR} they belong to $\Dm(\Gamma)$.
These circumstances imply that, in order to prove (\ref{mainprp1}), it
is enough to show that
$$
\lim_{\epsilon\downarrow 0}\Gamma_{2\epsilon}
(\Gamma_\epsilon R^+\cdot R^-)=
\Gamma(\Gamma R^+\cdot R^-),
$$
$$
\lim_{\epsilon\downarrow 0}\Gamma_{2\epsilon}
(R^+\cdot\Gamma_\epsilon R^-)=
\Gamma(R^+\cdot\Gamma R^-)
$$
in the strong operator topology. By Lemma \ref{lGmep} and
Proposition \ref{solRplmn}, $\Gamma_\epsilon R^+$ and
$\Gamma_\epsilon R^-$ are integral operators with the kernels
$$
\frac{r_\gamma(x)\overline{g(s)}\tilde\rho(s)}{x+\ir\epsilon-s}\quad
\mathrm{and}\quad
\frac{g(x)\overline{r_\gamma(s)}\tilde\rho(s)}{x+\ir\epsilon-s},
$$
respectively, where $r_\gamma(x)=\psi_\gamma(x)g(x)$ and
$\psi_\gamma(x)$ is defined
by (\ref{dfps}). Then $\Gamma_\epsilon R^+\cdot R^-$ and
$R^+\cdot\Gamma_\epsilon R^-$ are integral operators with the kernels
$$
r_\gamma(x)\int_\R\frac{|g(\tilde s)|^2\tilde\rho(\tilde s)\,d\tilde s}
{x+\ir\epsilon-\tilde s}\overline{r_\gamma(s)}\tilde\rho(s)\quad
\mathrm{and}\quad
r_\gamma(x)\int_\R\frac{|g(\tilde s)|^2\tilde\rho(\tilde s)\,d\tilde s}
{\tilde s+\ir\epsilon-s}\overline{r_\gamma(s)}\tilde\rho(s),
$$
respectively. Then, by Lemma \ref{lGmep}, $\Gamma_{2\epsilon}
(\Gamma_\epsilon R^+\cdot R^-)$ and $\Gamma_{2\epsilon}
(R^+\cdot\Gamma_\epsilon R^-)$ are integral operators with the kernels
$$
\frac{r_\gamma(x)}{x+2\ir\epsilon-s}\int_\R\frac{|g(\tilde s)|^2\tilde\rho(\tilde s)\,d\tilde s}
{x+\ir\epsilon-\tilde s}\overline{r_\gamma(s)}\tilde\rho(s)
$$
and
$$
\frac{r_\gamma(x)}{x+2\ir\epsilon-s}\int_\R\frac{|g(\tilde s)|^2\tilde\rho(\tilde s)\,d\tilde s}
{\tilde s+\ir\epsilon-s}\overline{r_\gamma(s)}\tilde\rho(s),
$$
respectively. Then, making use of the operator $P_{+,\epsilon}$,
defined by (\ref{pepl}), we have:
\begin{eqnarray}\label{Gm2ep1}
\Gamma_{2\epsilon}
(\Gamma_\epsilon R^+\cdot R^-)=(2\pi\ir)^2JM_{r_\gamma\cdot
P_{+,\epsilon}(|g|^2\tilde\rho)}P_{+,2\epsilon}M_{\bar r_\gamma
\cdot\tilde\rho},
\end{eqnarray}
\begin{eqnarray}\label{Gm2ep2}
\Gamma_{2\epsilon}
(R^+\cdot\Gamma_\epsilon R^-)=(2\pi\ir)^2JM_{r_\gamma}P_{+,2\epsilon}
M_{\overline{r_\gamma\cdot P_{+,\epsilon}(|g|^2\tilde\rho)}\tilde\rho}.
\end{eqnarray}
Recall that $\vert g\vert^2\tilde\rho=\tilde\rho$, because $g(x)=1$ almost
everywhere with respect  to the measure $\rho$. In view of
BTB property of $\rho$, the family of functions
$\{P_{+,\epsilon}\tilde\rho(x)\}_{\epsilon>0}$ is uniformly bounded on $\R$.
On the other hand, by the property of Hardy class $H_2$, the limiting
relation (\ref{limPplep}) is valid for almost all $x\in\R$. These
circumstances and the fact that $r_\gamma\in L_\infty(\R)$ for
a small enough $\gamma$, imply that for such $\gamma$
$$
\lim_{\epsilon\downarrow 0}M_{r_\gamma\cdot
P_{+,\epsilon}(|g|^2\tilde\rho)}=M_{r_\gamma\cdot
P_+(|g|^2\tilde\rho)}
$$
with respect  to the strong operator topology. Furthermore, by Proposition 2.1
of \cite{Ar-Z1}, $\lim_{\epsilon\downarrow 0}P_{+,\epsilon}=P_+$ with respect 
to this topology. Then we obtain from (\ref{Gm2ep1}) and
(\ref{Gm2ep2}) that
\begin{eqnarray}\label{lmGm2ep1}
\lim_{\epsilon\downarrow 0}\Gamma_{2\epsilon}
(\Gamma_\epsilon R^+\cdot R^-)=(2\pi\ir)^2JM_{r_\gamma\cdot
P_+(|g|^2\tilde\rho)}P_+M_{\bar r_\gamma
\cdot\tilde\rho},
\end{eqnarray}
and
\begin{eqnarray}\label{lmGm2ep2}
\lim_{\epsilon\downarrow 0}\Gamma_{2\epsilon}
(R^+\cdot\Gamma_\epsilon R^-)=(2\pi\ir)^2JM_{r_\gamma}P_+
M_{\overline{r_\gamma\cdot P_+(|g|^2\tilde\rho)}\tilde\rho}.
\end{eqnarray}
with respect to the strong operator topology. On the other hand, we can show
with the help of Proposition \ref{prGmR} that the right hand sides of
(\ref{lmGm2ep1}) and (\ref{lmGm2ep2}) are equal to
$\Gamma(\Gamma R^+\cdot R^-)$ and $\Gamma(R^+\cdot\Gamma  R^-)$
respectively. The lemma is proven.
\end{proof}

We now turn to main result of this subsection.

\begin{proposition}\label{invUpl}
Assume that the spectral measure $\rho$ of the operator $A$ has the
BTB property. Then for a small enough $\gamma\in\Cc$
$$
U_\gamma^+U_\gamma^-=U_\gamma^-U_\gamma^+=I.
$$
\end{proposition}
\begin{proof}
Using Lemma \ref{prpGm} and (\ref{eqUplmn}), we have:
\begin{eqnarray}
&&U_\gamma^+U_\gamma^--I=(I+\Gamma R^+)(I-\Gamma R^-)-I=
\Gamma(R^+-R^-)-\Gamma R^+\cdot\Gamma R^-=
\nonumber\\
&&\Gamma(R^+-R^--\Gamma R^+\cdot R^--R^+\cdot\Gamma R^-)=
\Gamma(R^+U_\gamma^--U_\gamma^+R^-).
\nonumber
\end{eqnarray}
But, in view of (\ref{eqRpl}) and (\ref{eqRmn}),
$R^+U_\gamma^-=U_\gamma^+BU_\gamma^-=U_\gamma^+R^-$.
Hence we have proved that $U_\gamma^+U_\gamma^-=I$, that is
$U_\gamma^-$ is a right inverse of the operator $U_\gamma^+$.
On the other hand, we obtain from (\ref{GmRpl}) and (\ref{rpleqrmn})
that
$$
\Gamma R^+=-2\pi\ir\gamma JM_{\psi_\gamma\cdot g}
P_+M_{\bar g\tilde\rho}.
$$
where $\psi_\gamma$ is defined by (\ref{dfps}).
Then we see that $\Vert\Gamma R^+\Vert<1$ for a small enough $\gamma$,
hence the operator $U_\gamma^+=I+\Gamma R^+$ has a bounded inverse
$(U_\gamma^+)^{-1}$, which coincides with the right inverse
$U_\gamma^-$. This means that $U_\gamma^-U_\gamma^+=I$. The proposition
is proven.
\end{proof}

\subsection{Final results}

On the base of previous results we can obtain
criterion for HUE property for the pair $A,\;B$ in the case
of the rank-one perturbation $B$ of the multiplication operator $A$.
The following result is valid.

\begin{theorem}\label{tuneq}
If the spectral measure $\rho$ of the operator $A$ has the
BTB property, then there exists $\delta>0$ such that

\noindent (i) the family
of operators $U_\gamma^+$, defined by (\ref{frmUpl}), depends
holomorphically on $\gamma$ in the open disk
$D_\delta=\{\lambda\in\Cc\;\vert\;|\lambda|<\delta\}$;

\noindent (ii) for each fixed $\gamma\in(-\delta,\delta)$ the operator $U_\gamma^+$
is unitary and it establishes a unitary equivalence between $A$
and $A_\gamma$, that is
\begin{equation}\label{uniteq}
A_\gamma=(U_\gamma^+)^{-1}AU_\gamma^+.
\end{equation}
In other words, the pair of operators $A,\;B$ has the HUE property.
\end{theorem}
\begin{proof}
Assertion (i) follows immediately from (\ref{frmUpl}) and (\ref{dfps}).
Let us prove assertion (ii). We see from (\ref{frmUpl}),
(\ref{frmUmn}) and (\ref{dfps})
that for a small enough $\gamma\in\R$ $(U_\gamma^+)^\star=U_\gamma^-$.
On the other hand, by Proposition \ref{invUpl}, for a small enough
$\gamma$
\begin{equation}\label{Umn}
U_\gamma^-=(U_\gamma^+)^{-1}.
\end{equation}
These circumstances mean that for a small enough $\gamma\in\R$ the
operator $U_\gamma^+$ is unitary. In view of (\ref{eqpl}) and
(\ref{Umn}), the formula (\ref{uniteq}) is valid. The theorem is proven.
\end{proof}

\begin{corollary}\label{cntfnct}
If the condition of Theorem \ref{tuneq} is satisfied, then for any
function $\phi\in C[-M,M]$ and for a small enough $\gamma$
the representation is valid:
\begin{eqnarray}\label{rprphiAgm}
&&\phi(A_\gamma)=U_\gamma^-M_{\phi_0}U_\gamma^+=
\left(I+2\pi\mathrm{i}\gamma JM_gP_+
M_{\bar\psi_\gamma\cdot\bar g\cdot\tilde\rho}\right)M_{\phi_0}\times
\nonumber\\
&&\left(I-2\pi\mathrm{i}\gamma JM_{\psi_\gamma\cdot g}P_+
M_{\bar g\cdot\tilde\rho}\right),
\end{eqnarray}
where $\psi_\gamma$ is defined by (\ref{dfps}) and $\phi_0$ is
the continuation of the function $\phi$ by zero out of the interval
$[-M,M]$. Recall that $M>0$ is such a number that
$\p(\rho)\subset(-M,M)$.
\end{corollary}
\begin{proof}
Since $A$ is the multiplication operator by the independent variable
in the space $\B=L_2(\R,\rho)$, then $\phi(A)=M_{\phi_0}$. Hence
we have from (\ref{uniteq}) that
$\phi(A_\gamma)=(U_\gamma^+)^{-1}M_{\phi_0}U_\gamma^+$
for a small enough $\gamma$. Then, taking
into account (\ref{frmUpl}), (\ref{frmUmn}) and (\ref{Umn}), we obtain
the desired equality (\ref{rprphiAgm}).
\end{proof}

As a consequence of Theorem \ref{tuneq} and Corollary \ref{cntfnct}
we obtain the following strengthening of Theorem 2.4 from \cite{Ar-Z1}:

\begin{corollary}\label{crholom}
If the spectral measure $\rho$ of the operator $A$ has the
BTB property, then

\noindent (i) the pair $A,\;B$ has the FH property;

\noindent (ii) the formula is valid:
\begin{equation}\label{frmdf}
\frac{\partial}{\partial\gamma}\phi(A_\gamma)\vert_{\gamma=0}=
\pi\ir M_g[H,M_{\phi_0}]
M_{\bar g\cdot\tilde\rho},
\end{equation}
where $H$ is the Hilbert transform in $L_2(\R)$:
$H=P_+-P_-\;(P_-=I-P_+)$.
\end{corollary}
\begin{proof}
Assertion (i) follows from Theorem \ref{tuneq} and Proposition
\ref{pHUEimpl}. Let us prove assertion (ii). We have from
(\ref{rprphiAgm}) and (\ref{dfps}):
$$
\phi(A_\gamma)-\phi(A)=2\pi\ir\gamma M_g(P_+M_{\phi_0}-M_{\phi_0}P_+)
M_{\bar g\cdot\tilde\rho}+O(\gamma^2).
$$
The latter representation and the formula $P_+=\frac{I+H}{2}$ imply
the desired formula (\ref{frmdf}).
\end{proof}

 The following assertion connects the CLI property
with BTB property in the case of the rank-one perturbation $B$ of
the multiplication operator $A$:

\begin{theorem}\label{tHrd}
If the pair of operators $A,\;B$, defined above, has the CLI
property, then the spectral measure $\rho$ of the operator $A$ has
the BTB property.
\end{theorem}
\begin{proof}
As was shown in \cite{Ar-Z}, a number
$\lambda\notin\sigma(A)$ is an eigenvalue of the
operator $A_\gamma$ if and only if it is a root of the equation
\begin{equation}\label{roots}
1+\gamma \mathrm{B}\rho(\lambda)=0.
\end{equation}
Recall that $\mathrm{B}\rho$ is the Borel transform of the spectral
measure $\rho$ of the operator $A$, defined by \eqref{dfBortr}. Assume, on the contrary, that
the measure $\rho$ does not have the
BTB property. This means that the function $\mathrm{B}\rho$
does not belong to the Hardy class $H_\infty$. Then there
exists such a sequence
$\lambda_n\in\Cc$ that $\Im(\lambda_n)>0$ and
$$
\lim_{n\rightarrow\infty}\mathrm{B}\rho(\lambda_n)=\infty.
$$
Denote $\gamma_n=-(\mathrm{B}\rho(\lambda_n))^{-1}$.
Then $\lim_{n\rightarrow\infty}\gamma_n=0$ and each $\lambda_n$ is a
root of the equation (\ref{roots}) with $\gamma=\gamma_n$.  Since
$\Im(\lambda_n)>0$, then $\lambda_n\notin\sigma(A)$, hence
each $\lambda_n$ is an eigenvalue of the operator $A_{\gamma_n}$.
 These circumstances contradict the CLI property of the pair $A,\;B$.
The theorem is proven.
\end{proof}

Proposition \ref{pHUEimpl} and Theorems \ref{tFHLSC}, \ref{tuneq}
and \ref{tHrd} imply the
following final result in the case of the rank-one perturbation $B$ of the
multiplication operator $A$:

\begin{theorem}\label{tfin}
If $A,\;B$ are the operators defined above, then the following
assertions are equivalent to each other:

\noindent (i) The pair $A,\;B$ has the HUE property;

\noindent (ii) The pair $A,\;B$ has the FH property;

\noindent (iii) The pair $A,\;B$ has the CLI property;

\noindent (iv) The spectral measure
$\rho$ of the operator $A$ has the BTB property.
\end{theorem}


\begin{thebibliography}{15}

\bibitem{Ar-Z1} J. Arazy and L. Zelenko, \textit{Directional operator differentiability
	of non-smooth functions,}
J. Operator Theory, \textbf{55}:1
(2006), 49-90

\bibitem{Ar-Z} J. Arazy and L. Zelenko, \textit{Finite-dimensional
	perturbations of self-adjoint operators,} Integral Equations
and Operator
Theory, \textbf{34} (1999), 127-164.

\bibitem{D-K} Yu. L. Daletskii and M. G. Krein,
\textit{The stability of solutions of
differential equations in Banach space}, Moscow, Nauka, 1970.

\bibitem{D-Sc} N. Dunford and J. T. Schwartz,
\textit{Linear Operators, Part I: General Theory},
Interscience Publishers, New York, London 1958.

\bibitem{F} K. O. Friedrichs, \textit{Perturbation of
Spectra in Hilbert Space}, Providence, R. I. : American Math. Soc.,
1965.

\bibitem{Kat} T. Kato, \textit{Perturbation Theory of Linear
Operators}, Springer-Verlag, New York, Tokyo 1984.

\bibitem{Rei} P.A. Rejto, \textit{On gentle perturbations,} Comm. Pure Appl. Math. vol. \textbf{XVI} (1963), 279-303.



\end{thebibliography}
\end{document}